\documentclass[11 pt,leqno]{amsart}

\usepackage{amsmath}
\usepackage{amssymb}
\usepackage{amsfonts}
\usepackage{amsthm}
\usepackage{hyperref}
\usepackage{enumerate}

\usepackage{bbm}
\usepackage[mathscr]{eucal}

\makeatletter
\def\subsection{\@startsection{subsection}{2}%
  \z@{.5\linespacing\@plus.7\linespacing}{-.5em}%
  {\normalfont\itshape}}
\makeatother

\makeatletter
\def\author@andify{%
  \nxandlist{\unskip\penalty-1\;\;\space\ignorespaces}%
           {\unskip\penalty-1\;\;\space\ignorespaces}%
           {\unskip\penalty-2\;\;\space\ignorespaces}%
}

\makeatother

\usepackage{tikz}
\usetikzlibrary{intersections}
\usetikzlibrary{calc}
\usepackage{bbm}
\usepackage{color}\usepackage{enumitem}
\usepackage[mathscr]{eucal}

\theoremstyle{plain}
\newtheorem{thm}{Theorem}[section]
\newtheorem{lemma}[thm]{Lemma}

\newtheorem{proposition}[thm]{Proposition}

\theoremstyle{remark}
\newtheorem{remark}[thm]{Remark}

\numberwithin{equation}{section}

\newcommand{\inn}[2]{\langle#1,#2\rangle}

\newcommand{\Pthreex}[1]{(\d-1)/(\d-1+#1)}
\newcommand{\Pthreey}{0}
\newcommand{\Pfourx}[1]{(\d-1)/(\d-1+#1)}
\newcommand{\Pfoury}[1]{(1-\d)/(\d-1+#1)}

\newcommand{\Aonex}{(\d-1)/(\d)}
\newcommand{\Aoney}{(1-\d)/(\d)}
\newcommand{\Atwox}[1]{(\d-1)/(\d-1+#1)}
\newcommand{\Atwoy}[1]{(1-\d)/(\d-0.75+#1)}

\newcommand{\definecoords}{
\def\ptsize{.1pt}
\def\WbgFillOpacity{.2}
\def\WbgFillColor{gray}
\def\WbgLineOpacity{.6}
\def\WbgDrawCritSeg{1}
\def\WbgCritSegStyle{solid}

\def\WbbFillOpacity{.1}
\def\WbbLineOpacity{.5}
\def\AFillColor{white}
\def\AFillOpacity{0}
\def\ALineOpacity{.8}
\coordinate (P1) at (0,0);
\coordinate (P2) at (0,1);
\coordinate (P3) at ( {\Pthreex{\b}}, { \Pthreey } );
\coordinate (P4g) at ( { \Pfourx{\b} }, { \Aoney } );
\coordinate (P4) at ( { \Pfourx{\b} }, { \Pfoury{\b} } );
\coordinate (A1) at ( { \Aonex }, { \Aoney } );
\coordinate (A2) at ( { \Atwox{\b} }, { \Atwoy{\b} } );
\coordinate (A3) at (1,0);
\coordinate (A4) at (0,1);
\coordinate (C1) at ( { (\Pfourx{\b}+2*\Aonex)/3 }, {(\Pfoury{\b}+2*\Aoney)/3} );
\coordinate (C2) at ( { \Pfourx{\b} }, {(\Pfoury{\b}+\Aoney)/2} );
}

\newcommand{\drawauxlines}[1]{
\draw (0,-1) [->] -- (0,1.2) node [left] {\tiny $\frac{\alpha}{(d-1)p}$};
\draw (0,0) [->] -- (1.2,0) node [below] {\tiny $\frac1p$};
\draw[dashed,opacity=.3] (0,0) -- (1,-1);
\draw[dashed,opacity=.3] (0,1) -- (#1);
\draw[dashed,opacity=.3] ({\Aonex},1) -- ({\Aonex},-1);
\draw[dashed,opacity=.3] ({\Pthreex{\b}},1) -- ({\Pthreex{\b}},-1);
\draw[dashed,opacity=.3] (0,-{\Aonex}) -- (1,-{\Aonex});
\draw[dashed,opacity=.3] (0,-{\Pthreex{\b}}) -- (1,-{\Pthreex{\b}});
}

\newcommand{\drawWbg}{
\fill ({\Aonex},0) node [below left] {\tiny $x_1$} circle [radius=0.015em];
\fill (P3) node [below right] {\tiny $x_\beta$} circle [radius=0.015em];
\fill (A3) node [below] {\tiny $1$} circle [radius=0.015em];
\fill (A4) node [left] {\tiny $1$} circle [radius=0.015em];
\fill (0,-{\Aonex}) node [left] {\tiny $-x_1$} circle [radius=0.015em];
\fill (0,-{\Pthreex{\b}}) node [left] {\tiny $-x_\beta$} circle [radius=0.015em];
\fill[color=\WbgFillColor,opacity=\WbgFillOpacity] (P1) -- (A1) .. controls (C1) and (C2) .. (A2) -- (P3) -- (P2) -- cycle;
\draw[opacity=\WbgLineOpacity] (P1) -- (A1) .. controls (C1) and (C2) .. (A2) -- (P3) -- (P2) -- cycle;
}

\newcommand{\drawA}{
\fill[color=\AFillColor,opacity=\AFillOpacity] (A1) .. controls (C1) and (C2) .. (A2) -- (P4) -- cycle;
\draw[opacity=\ALineOpacity] (A1) .. controls (C1) and (C2) .. (A2) -- (P4) -- cycle;
}

\begin{document}



\title[Power weight inequalities for spherical maximal functions]{Power weight inequalities for spherical maximal functions}

\author[ M. Fraccaroli ]{Marco Fraccaroli}
\address{M.F., Department of Mathematics and Statistics, University of Massachusetts Lowell, Lowell, MA, USA}
\email{marco\_fraccaroli@uml.edu}

\author[ J. Roos ]{Joris Roos}
\address{J.R., Department of Mathematics and Statistics, University of Massachusetts Lowell, Lowell, MA, USA}
\email{joris\_roos@uml.edu}

\author[ A. Seeger ]{Andreas Seeger}
\address{A.S., Department of Mathematics, University of Wisconsin--Madison, Madison, WI, USA}
\email{aseeger@wisc.edu}

\thanks{J.R. was supported in part by NSF grant DMS-2154835 and a Simons Foundation grant.
A.S. was supported in part by NSF grant DMS-2348797.}

\begin{abstract}
This paper is about
spherical maximal functions with general dilation sets acting on functions in weighted
$L^p(|x|^\alpha)$ spaces. Aside from endpoint cases, a complete description of the allowable ranges of $p$, $\alpha$ is given in terms of
the Legendre--Assouad function of the dilation set. This settles, up to endpoints, an open problem of Duoandikoetxea and Seijo.
\end{abstract}

\maketitle

\newcommand{\todo}[1]{{\color{blue}ToDo: #1}}

\section{Introduction}
Let $d\ge 2$. For a locally integrable function $f:{{\mathbb{R}}}^d\to {{\mathbb{C}}}$ the average over the sphere $S^{d-1}$ of radius $t$ centered at $x\in {{\mathbb{R}}}^d$ is given by
\[A_tf(x) =f*\sigma_t(x)=\int f(x+t{\omega}) d\sigma({\omega}),\]
where $\sigma$ is the normalized surface measure and the dilate $\sigma_t$ is defined by $\inn{\sigma_t}{f}=\inn{\sigma}{f(t\cdot)}$. The average $A_tf(x)$ is well-defined for almost every $x\in {{\mathbb{R}}}^d$. Given a set ${{\mathcal{E}}}$ of radii in $(0,\infty)$, the maximal function $M_{{\mathcal{E}}}$ is given by
\[M_{{\mathcal{E}}} f(x)= \sup_{t\in {{\mathcal{E}}}} |A_t f(x)|\]
which yields a measurable function at least for continuous $f$. If ${{\mathcal{E}}}$ has an accumulation point at $0$, then size estimates for the maximal operator can be used to prove pointwise convergence results for $A_t f(x)$ as $t\to 0$ within ${{\mathcal{E}}}$.
For the full spherical maximal operator, corresponding to ${{\mathcal{E}}}=(0,\infty)$, $L^p({{\mathbb{R}}}^d)$ boundedness holds if and only if $p>1+\frac{1}{d-1}$; this is due to Stein \cite{SteinPNAS1976} for $d\ge 3$ and to Bourgain \cite{BourgainJdA1986} for $d=2$.

The case of general ${{\mathcal{E}}} \subset (0,\infty)$ was considered in \cite{SeegerWaingerWright1995}, where it was shown that the boundedness range depends on a variant of Minkowski dimension.
Equip $(0,\infty)$ with the metric
\begin{equation}\label{eqn:metric}
d_\times(s,t) = |\log_2 (s/t)|.
\end{equation}
We denote the diameter of a set $J\subset (0,\infty)$ by $|J|_{\times}$.
Note that dilations $t\mapsto \lambda t$ with $\lambda>0$ are isometries on $(0,\infty)$, so that e.g. $|[R,2R]|_{\times}=1$ for all $R>0$.
For a set $E\subset (0,\infty)$ with bounded diameter we denote by $N(E,\delta)$ the {\em entropy numbers}, i.e. for each $\delta>0$ the minimum number of intervals of diameter $\delta$ required to cover $E$ (where diameter is always taken with respect to $d_\times$).

With ${{\mathcal{E}}}\subset (0,\infty)$ fixed, define
\[ \beta = \beta_{{\mathcal{E}}} = \limsup_{\delta\to 0} \frac{\sup_{|J|_\times= 1} \log N({{\mathcal{E}}}\cap J, \delta)}{\log (\delta^{-1}) }, \]
where the supremum is taken over intervals $J\subset (0,\infty)$ of diameter one.
This is a natural extension of the standard (upper) Minkowski dimension to the case of unbounded sets; indeed, $\beta$ is equal to the Minkowski dimension if ${{\mathcal{E}}}$ has bounded diameter (with respect to $d_\times$).
The result in \cite{SeegerWaingerWright1995} states that
$M_{{\mathcal{E}}} $ is bounded on $L^p$ if
\[p>p_\beta=1+\tfrac{\beta}{d-1}\]
and unbounded if $p<p_\beta$
(see also \cite{STW-jussieu}, \cite{RoosSeeger-problems} for the discussion of various endpoint problems).

In this paper we are interested in power weight inequalities of the form
\begin{equation} \label{eq:weighted}
\int |M_{{\mathcal{E}}} f(x)|^p |x|^\alpha\;dx {\lesssim} \int |f(y)|^p |y|^\alpha\;dy.
\end{equation}
Thus for the weight $w_\alpha(x)=|x|^\alpha$ we ask for which $p\in [1,\infty]$ and which $\alpha\in {{\mathbb{R}}}$ the sublinear operator $M_{{\mathcal{E}}}$ maps $L^p(w_\alpha)$ to itself.
Our main result answers this question for every given ${{\mathcal{E}}}\subset (0,\infty)$ up to endpoint cases in the parameters $p$ and $\alpha$.

More precisely, we will determine the closure of the ``type set''
\begin{equation}\label{eq:fW}
{{\mathfrak{T}}}_{{\mathcal{E}}}=\{ (\tfrac 1p, \tfrac \alpha p)\in [0,1]\times \mathbb{R} \;: \; M_{{\mathcal{E}}} \text{ is bounded on } L^p(w_\alpha) \},
\end{equation}
which is in one-to-one correspondence with the closure of the set of all $(\tfrac 1p,\alpha)$ for which $M_{{\mathcal{E}}}$ is bounded on $L^p(w_\alpha)$.
Observe that ${{\mathfrak{T}}}_{{\mathcal{E}}}$ is convex (for the short proof see \S \ref{sec:typesetconvex}).

To formulate our results we need
another quantification of dimension besides $\beta$, called the {\it Legendre--Assouad function} which was first introduced in \cite{BeltranRoosSeeger-radial, BeltranRoosRutarSeeger} for subsets of $[1,2]$.
For $\rho\in\mathbb{R}$ we let
\begin{equation}\label{eq:LAglob}
\nu^\sharp (\rho)= \nu^\sharp_{{\mathcal{E}}}(\rho)= \limsup_{\delta\to 0} \frac{\log\big(\sup_{\delta\le |J|_\times\le 1} |J|_\times^{-\rho} N({{\mathcal{E}}}\cap J, \delta)\big)}{ \log(\delta^{-1})},
\end{equation}
where the supremum is taken over intervals $J\subset (0,\infty)$ of $d_\times$-diameter in $[\delta,1]$.
This is a natural extension of the original definition from \cite{BeltranRoosSeeger-radial, BeltranRoosRutarSeeger} to unbounded sets.
One can show (\cite[\S2]{BeltranRoosRutarSeeger}) that $\nu^\sharp$ equals the Legendre transform of the function
\[\nu(\theta)=-(1-\theta)\dim_{\mathrm A,\theta}{{\mathcal{E}}},\]
where $\theta\mapsto \dim_{\mathrm A,\theta}{{\mathcal{E}}}$ is the Assouad spectrum of ${{\mathcal{E}}}$ with respect to the metric $d_\times$ (see Fraser--Yu \cite{FraserYu2018Adv}, Fraser \cite{FraserBook}).

Recall from \cite{BeltranRoosRutarSeeger} that $\nu^\sharp$ is convex and increasing, and $\nu^\sharp(\rho)=\beta$ for $\rho\le 0$.
Now consider the generalized inverse of the increasing function $\nu^\sharp$ which is defined for $s\ge \beta$ as
\[(\nu^\sharp)^\dagger(s) = \sup \{ \rho\ge 0\,:\,\nu^\sharp(\rho)\le s\}.\]
For $p\ge p_\beta=1+\frac{\beta}{d-1}$ let
\begin{align} U(p)&= (d-1)(p-1)-\beta,
\\L(p) &=(d-1)(p-2)- (\nu^\sharp)^{\dagger} \big((d-1)(p-1) \big).
\end{align}

\begin{thm}\label{thm:main}
Let $d\ge 2$, ${{\mathcal{E}}}\subset (0,\infty)$ nonempty. Then
\begin{equation}\label{eq:typesetexplicit} \overline{{{\mathfrak{T}}}_{{\mathcal{E}}}} = \{ \big(\tfrac 1p,\tfrac\alpha p\big )\in [0,1]\times \mathbb{R}\,:\, p\ge p_\beta,\,\, L(p)\le \alpha\le U(p)\}.
\end{equation}
\end{thm}
Figure \ref{fig:WE} visualizes a typical set $\overline{{{\mathfrak{T}}}_{{\mathcal{E}}}}$, for a case with $\gamma=1$. To avoid a large picture we have scaled the vertical axis by the factor $(d-1)^{-1}$.
Note that $\overline{{{\mathfrak{T}}}_{{\mathcal{E}}}}$ is not necessarily a polygon, since $\nu^\sharp|_{[0,\infty)}$ can be any nonnegative increasing convex function such that $\nu^\sharp(\rho)=\rho$ for $\rho\ge 1$ (see \cite[Thm. 1.2 (ii)]{BeltranRoosRutarSeeger}).

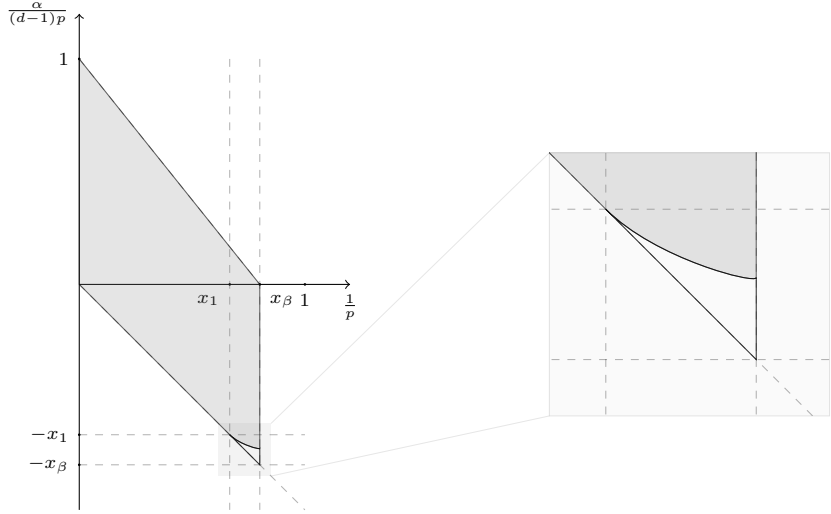
\begin{figure}[ht]
\begin{tikzpicture}
\def\d{3}
\def\b{.5}

\def\clipPad{.05}
\def\clipLx{ \Aonex - \clipPad }
\def\clipLy{ \Pfoury{\b} -\clipPad }
\def\clipHx{ \Pfourx{\b} + \clipPad }
\def\clipHy{ \Aoney + \clipPad }
\def\clipStyle{solid}
\def\clipOpacity{.1}

\begin{scope}[scale=3]
\definecoords

\coordinate (clipLL) at ({\clipLx}, {\clipLy});
\coordinate (clipHL) at ({\clipHx}, {\clipLy});
\coordinate (clipHH) at ({\clipHx}, {\clipHy});
\fill[opacity=.05] (clipLL) rectangle (clipHH);

\drawauxlines{P3}
\drawWbg
\drawA
\end{scope}

\begin{scope}[xshift=-3cm, yshift=11cm, scale=15]
\definecoords

\coordinate (clipLLw) at ({\clipLx}, {\clipLy});
\coordinate (clipLHw) at ({\clipLx}, {\clipHy});
\coordinate (clipHHw) at ({\clipHx+.015}, {\clipHy});
\draw[\clipStyle, opacity=\clipOpacity] (clipHL) -- (clipLLw);
\draw[\clipStyle, opacity=\clipOpacity] (clipHH) -- (clipLHw);
\draw[\clipStyle, opacity=\clipOpacity] (clipLLw) rectangle (clipHHw);
\fill[opacity=.02] (clipLLw) rectangle (clipHHw);
\clip (clipLLw) rectangle (clipHHw);

\drawauxlines{P3}
\drawWbg
\drawA

\end{scope}

\end{tikzpicture}

\caption{Typical shape of the type set ${{\mathfrak{T}}}_{{\mathcal{E}}}$ (shaded). The containing trapezoid is defined by the necessary conditions $p\ge p_\beta$ and \eqref{eqn:oldneccond}. Here
$x_1=\tfrac1{p_1}$, $x_\beta=\tfrac1{p_\beta}$.
}

\label{fig:WE}

\end{figure}

Note that if $\rho_*=\sup\{\rho: \nu^\sharp(\rho)=\beta\}$, then $0\le \rho_*\le \beta$, and by convexity of $\nu^\sharp$ it is continuous and strictly increasing for $\rho\ge \rho_*$. Thus $(\nu^\sharp)^\dagger$ is equal to the inverse of $\nu^\sharp|_{[\rho_*,\infty)}$.
Note also that $\nu^\sharp(\rho)=\rho$ for $\rho\ge \gamma$, where $\gamma\ge \beta$ is the quasi-Assouad dimension of ${{\mathcal{E}}}$ (see \cite{LuXi2016,FraserHareHareTroscheitYu,FraserBook}).
Finally, note $L(p_\beta)= 1-d+\beta-\rho_*$ and
$L(p)=1-d$ for $p\ge p_\gamma$.

We may also state Theorem \ref{thm:main} with a more concise, albeit implicit condition.
Namely, if we set
\begin{equation}\label{eq:Thetadef}\Theta(p,\alpha)= \max\big \{\alpha+\beta, \,\,\nu^\sharp\big((d-1)(p-2)-\alpha\big)\big\},
\end{equation}
then
\begin{equation}\label{eq:typesetTheta} \overline{{{\mathfrak{T}}}_{{\mathcal{E}}}} = \big\{ (\tfrac 1p,\tfrac\alpha p)\in [0,1]\times \mathbb{R} : p\ge 1+ (d-1)^{-1}\Theta(p,\alpha)\big\}.
\end{equation}
See \S \ref{sec:Thetaequivalence} for details on this equivalence.
Our proof of Theorem \ref{thm:main} will rely on the implicit formulation.

We briefly discuss Theorem \ref{thm:main} in the context of the previous literature.
In the unweighted case $\alpha=0$ we recover the result from \cite{SeegerWaingerWright1995}. Examples from the unweighted theory show that the condition
$p\ge p_\beta$ is necessary. Since $\nu^\sharp (\rho)\ge \max(\rho, \beta)$ for all $\rho\in {{\mathbb{R}}}$ (see \cite[Lemma 2.1]{BeltranRoosSeeger-radial}), we get the necessary condition \begin{equation}\label{eqn:oldneccond}
-(d-1)\le \alpha\le (d-1)(p-1)-\beta,
\end{equation}
which was also known previously.

The case of the full maximal operator with ${{\mathcal{E}}}=(0,\infty)$ was almost fully resolved by Duoandikoetxea and Vega \cite{DuoandikoetxeaVega}, where it was shown that boundedness holds for the full maximal operator with ${{\mathcal{E}}}=(0,\infty)$ when $1-d<\alpha<(d-1)p-d$ and $p>1+\frac{1}{d-1}$. This matches our condition since in this case $\nu^\sharp(\rho)=\max(1,\rho)$.
See also the recent endpoint results for $\alpha=1-d$ by Juyoung Lee \cite{LeeJ-weighted}, for $d=2$, $p>2$, and $d\ge 3$, $p\ge 2$.
Sharp results for the full spherical maximal operator acting on radial functions can be found in \cite{DuoandikoetxeaMoyuaOruetxebarria}, see also \cite{NowakRoncalSzarek}.
In the lacunary case ${{\mathcal{E}}}=\{2^k:k\in {{\mathbb{Z}}}\}$ it was proved in \cite{DuoandikoetxeaVega} that $M_{{{\mathcal{E}}}}$ is bounded on $L^p(w_\alpha)$ if and only if $1-d\le \alpha<(d-1)(p-1)$, $p>1$.

For more general sets ${{\mathcal{E}}}\subset (0,\infty)$ the maximal operators $M_{{\mathcal{E}}}$ were investigated by Duoandikoetxea and Seijo \cite{DuoandikoetxeaSeijo}.
They showed that $M_{{\mathcal{E}}}$ fails to be $L^p(w_\alpha)$-bounded if $\alpha<1-d$ or if $\alpha> (d-1) (p-1)-\beta$, and proved a positive result if $1-d<\alpha< (d-1) (p-1)-\beta$, and $p>p_1=1+\frac{1}{d-1}$. Theorem \ref{thm:main} shows that the latter $p$-range can be improved to $p>p_{\gamma}=1+\frac{\gamma}{d-1}$.

Less was known in the remaining cases
$p_\beta \le p\le p_\gamma$. The zoomed-in region in Figure \ref{fig:WE} shows a typical behavior in this range.
For a class of ``regular sets'' ${{\mathcal{E}}}$ an essentially sharp result for this range was obtained in \cite{DuoandikoetxeaSeijo}. In our terminology, these are precisely the sets where $\beta$ is equal to the Assouad dimension
and in this case $\nu^\sharp(\rho)=\max(\rho,\beta)$; specific examples include self-similar Cantor sets.
Duoandikoetxea and Seijo \cite{DuoandikoetxeaSeijo} also showed that there are sets $E_1, E_2$ with $\beta_{E_1}=\beta_{E_2}$ but with a different $\alpha$-range for $L^p(w_\alpha)$ boundedness for certain values of $p$. In particular, they investigated the model examples ${{\mathcal{E}}}=\{1+n^{-a}\,:\, n\in \mathbb{N}\}$, proved an essentially sharp result for $a=1$, but obtained only partial results for $a\neq 1$. These sets are all special cases
of \emph{Assouad regular sets}; for the definition we refer to \S\ref{sec:examples}. There we further explore our boundedness condition for Assouad regular sets, and then also for finite unions of those. However,
Theorem \ref{thm:main} goes far beyond these classes and settles the problem on $L^p(w_\alpha)$-boundedness, up to endpoints, for {\it all} dilations sets.

A key point in the proof is the discussion of two families of examples, one of them having appeared previously in papers on $L^p$-improving bounds for spherical maximal operators in \cite{AHRS, RoosSeeger}. These yield the required lower bounds on the maximal function proved in \S \ref{sec:lowerbounds}. The upper bounds for the proof of Theorem \ref{thm:main} are given in \S\ref{sec:upperbounds}.

\subsection*{Problems for general weights} One may raise the question of weighted $L^p$ inequalities with {\it general} weights; we thank the anonymous referee for the suggestion to discuss it here. A characterization of weights $w$ and exponents $p$ for which $M_{{\mathcal{E}}}$ is bounded on $L^p(w)$ is currently wide open, for all choices of ${{\mathcal{E}}}$. Weighted norm inequalities for some classes of general weights satisfying Muckenhoupt and reverse H\"older inequalities can be established via general implications of the theory of sparse domination, see \cite{bernicot-frey-petermichl}. We note that while sparse bounds are also interesting in their own right they do not currently recover the sharp $L^p(w_\alpha)$ boundedness ranges for power weights.
The sparse bounds for $M_{{\mathcal{E}}}$ are in turn closely connected to unweighted $L^p$-improving bounds for local versions of the maximal operator, see \cite{laceyJdA19}, \cite{BeltranRoosSeeger-sparse}, and \cite{AHRS, RoosSeeger}.
A characterization of the $L^p$ improving bounds for the local operators $M_E$, for $E\subset[1,2]$ is essentially known only for finite unions of Assouad-regular sets, and the $L^p$ improving problem for general $E\subset [1,2]$ is conjectured to be closely related to the Legendre--Assouad function $\nu_E^\sharp$ discussed above. For a more precise and detailed account of these connections we refer the reader to the survey paper \cite{RoosSeeger-problems}.

\subsection*{Notation} We write, for nonnegative quantities $u,v$, $u \lesssim v$ or $u\lesssim_L v$
to indicate $u \leq C v$ for some constant $C$ which may depend on some list $L$. We write
$u \approx v$ to indicate that both $u \lesssim v$ and $v \lesssim u $ hold.

\section{Preliminaries}

\subsection{Type set convexity}\label{sec:typesetconvex}
We show that the type set ${{\mathfrak{T}}}_{{\mathcal{E}}}$ as defined in \eqref{eq:fW} is convex.
Consider the operator
$\widetilde M_{u,{{\mathcal{E}}}}$ defined by \[\widetilde M_{u,{{\mathcal{E}}}} f(x)= w_u(x) M_{{\mathcal{E}}}[ fw_{-u}](x)\]
and note that the $L^p(w_\alpha)$ boundedness of $M_{{\mathcal{E}}}$ is equivalent to the boundedness of $\widetilde M_{u,{{\mathcal{E}}}}$ on unweighted $L^p$ if $u=\alpha/p$. Hence the convexity of ${{\mathfrak{T}}}_{{\mathcal{E}}}$ can be seen by applying Stein's analytic interpolation theorem \cite{stein-weiss}
to the family of linear operators
$T_{u, {{\mathcal{E}}}} f(x)= |x|^{u} A_{t(x) } [ f|\cdot|^{-u}] (x)$
where $u= (1-z)u_0+zu_1$ with $(p_0^{-1}, u_0)\in {{\mathfrak{T}}}_{{\mathcal{E}}}$, $(p_1^{-1}, u_1)\in {{\mathfrak{T}}}_{{\mathcal{E}}}$, $0\le {\operatorname{Re\,}}(z)\le 1$, and $t(x)$ is a measurable function of radii with $t(x)\in{{\mathcal{E}}}$. Alternatively (and equivalently) one can also argue more directly
by applying the Stein--Weiss theorem for interpolation with changes of measures \cite{SteinWeiss-weights} to the operator $A_{t(\cdot)}$.

\subsection{An equivalent condition}\label{sec:Thetaequivalence}
We will prove the equivalence of the two type set descriptions \eqref{eq:typesetexplicit} and \eqref{eq:typesetTheta}. The condition $p\ge 1+ (d-1)^{-1}\Theta(p,\alpha)$
in \eqref{eq:typesetTheta} can be written as
\begin{equation}
\label{Thmod} \max\big\{ \alpha+\beta,\, \nu^\sharp\big( (d-1)(p-2)-\alpha\big) \big\} \le (d-1)(p-1) .
\end{equation}
Thus we need to check that \eqref{Thmod}
is equivalent with the conditions $p\ge p_\beta=1+\frac{\beta}{d-1}$, $L(p)\le \alpha\le U(p)$.
The condition $\alpha\le U(p)$ coincides with the condition
$\alpha+\beta \le (d-1)(p-1)$.
Thus it only remains to show that
\begin{equation} \label{eq:equivlower}
\nu^\sharp\big((d-1)(p-2)-\alpha\big) \le (d-1)(p-1) \iff p\ge p_\beta,\,\, \alpha\ge L(p) .
\end{equation}
Recall that by definition of $(\nu^\sharp)^\dagger(s)=\sup\{\rho\,:\,\nu^\sharp(\rho)\le s\}\ge 0$ for $s\in [\beta, \infty)$ and since $\nu^\sharp$ is continuous and increasing, we have for all $s\ge \beta$ and $\rho\in\mathbb{R}$:
\[ \nu^\sharp(\rho)\le s\quad \Longleftrightarrow\quad \rho\le (\nu^\sharp)^\dagger(s). \]
Setting $\rho=(d-1)(p-2)-\alpha$ and $s=(d-1)(p-1)$ and noting that $\nu^\sharp(\rho)\le s$ implies $\beta\le s$ (i.e. $p\ge p_\beta$), we obtain \eqref{eq:equivlower}.

\subsection{Scaling and a global to local reduction} \label{sec:localglobal}
It will be useful to observe that
the maximal operator
enjoys dilation invariance. Precisely, for every $\lambda>0$ and every ${{\mathcal{E}}}\subset (0,\infty)$ we have
\begin{equation}\label{eq:scaling} \|M_{\lambda {{\mathcal{E}}}}\|_{L^p(w_\alpha)\to L^p(w_\alpha) } = \|M_{{{\mathcal{E}}}}\|_{L^p(w_\alpha)\to L^p(w_\alpha)}.
\end{equation}
Moreover, Duoandikoetxea and Seijo \cite[Lemma 4]{DuoandikoetxeaSeijo} observed that for $\alpha<0$ the global bounds for $M_{{\mathcal{E}}}$ follow from the uniform bounds for the spherical maximal operators associated to the sets ${{\mathcal{E}}}_R = R^{-1}({{\mathcal{E}}} \cap[R,2R])\subset [1,2]$, provided we assume $L^p$-boundedness for the unweighted operator.
\begin{lemma}[\cite{DuoandikoetxeaSeijo}] \label{lem:global}
Let $p\in [1,\infty]$.
Then for $\alpha< 0$ there exists a constant $C_{\alpha,p} <\infty $ such that
\[ \|M_{{{\mathcal{E}}}}\|_{L^p(w_\alpha)\to L^p(w_\alpha) } \le C_{\alpha,p} \big
(\|M_{{\mathcal{E}}}\|_{L^p(\mathbb{R}^d)\to L^p(\mathbb{R}^d)}+ \sup_{R>0} \|M_{{{\mathcal{E}}}_R}\|_{L^p(w_\alpha)\to L^p(w_\alpha) }\big).
\]
\end{lemma}

\section{Examples of type sets}\label{sec:examples}
In this section we provide further illustrations of special cases of Theorem \ref{thm:main}.

\subsection{Assouad regular sets}
We examine the boundedness condition in Theorem \ref{thm:main} for some additional examples of sets ${{\mathcal{E}}}$.
The Assouad spectrum of ${{\mathcal{E}}}\subset (0,\infty)$ (taken with respect to the appropriate metric \eqref{eqn:metric}) will be denoted by $\dim_{A,\theta}\,{{\mathcal{E}}}$. Note that it equals the Assouad spectrum taken with respect to the Euclidean metric whenever ${{\mathcal{E}}}$ is compactly supported in $(0,\infty)$.
We say that a set ${{\mathcal{E}}}$ is {\it $(\beta,\gamma)$-Assouad regular} (also just {\it(quasi)-Assouad regular}) if
\[\dim_{\mathrm{A},\theta} {{\mathcal{E}}}=
\min(\tfrac{\beta}{1-\theta},\gamma), \]
Here $\gamma$ denotes the quasi-Assouad dimension.
In this case the Legendre--Assouad function is given by
\[\nu_{{\mathcal{E}}}^\sharp(\rho)= \begin{cases}
(1-\tfrac \beta\gamma)\rho+\beta &\text{ if $\rho\le \gamma $, }
\\
\rho &\text{ if } \rho>\gamma.
\end{cases}
\]

This class
of sets played a special role in several previous papers, see e.g. \cite{AHRS,RoosSeeger,BeltranRoosSeeger-radial,BeltranRoosRutarSeeger,Rutar24}.
If ${{\mathcal{E}}}$ is $(\beta,\gamma)$-Assouad regular, then the boundedness condition of Theorem \ref{thm:main} can be summarized as follows:
\\

(i) If $p>p_\gamma=1+\frac{\gamma}{d-1}$, then $M_{{\mathcal{E}}}$ is bounded on $L^p(w_\alpha)$ if
\[ 1-d<\alpha<(d-1)(p-1)-\beta. \]
This part holds for all sets ${{\mathcal{E}}}$ with quasi-Assouad dimension $\gamma$ and also fully describes up to endpoints the boundedness in the case of the ``regular sets'' from \cite{DuoandikoetxeaSeijo} where $\beta=\gamma$.\\

(ii) If $p_\beta<p\le p_\gamma$, then $M_{{\mathcal{E}}}$ is bounded on $L^p(w_\alpha)$ if
\[1-d+ {\beta}\,\tfrac{\gamma-(d-1)(p-1)}{{\gamma}-\beta} <\alpha< (d-1)(p-1)-\beta.\]

These conditions are sharp up to endpoints.
The type set ${{\mathfrak{T}}}_{{\mathcal{E}}}$ in this case takes the shape of a polygon. If $\beta<\gamma$, then the lower boundary in the zoomed-in portion of Figure \ref{fig:WE} consists of a piecewise linear function with two pieces (the `kink' occurs at $1/p=1/p_\gamma$ and corresponds to the transition from case (i) to case (ii)).

\begin{remark}
The sequence sets $E_a=\{1+n^{-a}\,:\,n\in\mathbb{N}\}\subset [1,2]$ for $a>0$ are examples of Assouad regular sets with $\beta=\frac1{1+a}$ and $\gamma=1$, and so are the corresponding `global' sets ${{\mathcal{E}}}_a=\cup_{k\in\mathbb{Z}} 2^k E_a\subset (0,\infty)$.
The middle-third Cantor set in $[1,2]$ is an example of an Assouad regular set with $\beta=\gamma=\log_32$ (and in this case our result is already covered by \cite{DuoandikoetxeaSeijo}).
\end{remark}

\subsection{Finite unions}
The boundedness condition becomes more complicated if we consider finite unions of Assouad-regular sets.
Let ${{\mathcal{E}}}=\cup_{j=1}^N {{\mathcal{E}}}_j$ where ${{\mathcal{E}}}_j$ is a $(\beta_j,\gamma_j)$-Assouad regular set.
Then $\beta=\beta_{{{\mathcal{E}}}}$ equals $\max_{j=1,\dots,N} \beta_j$ and $M_{{\mathcal{E}}}$ is bounded on $L^p(w_\alpha)$ if
$ p>p_\beta $ and $L(p)<\alpha<(d-1)(p-1)-\beta,$ where
\begin{align*}
L(p)&= 1-d+ \max\Big\{0, \max_{\substack{j=1,\dots, N\\ {\gamma}_j > \beta_j} }
{\beta_j}\,\tfrac{\gamma_j-(d-1)(p-1)}{{\gamma}_j-\beta_j} \Big\}
\end{align*}
and this condition is sharp up to endpoints.
Since
\[\max_{j=1,\dots,N} \|M_{{{\mathcal{E}}}_j}\|_{p\to p} \le \|M_{{{\mathcal{E}}}}\|_{p\to p}\le \sum_{j=1}^N \|M_{{{\mathcal{E}}}_j}\|_{p\to p},\]
this is an immediate consequence of the Assouad regular case.
The type set ${{\mathfrak{T}}}_{{\mathcal{E}}}$ in this case again takes the form of a polygon with an increased number of edges on the lower boundary in the zoomed-in region of Figure \ref{fig:WE}.

\section{Lower bounds}\label{sec:lowerbounds}
In this section we prove the inclusion `$\subset$' in Theorem \ref{thm:main}.
The necessary condition $\alpha \le (d-1)(p-1)-\beta$ can be found in the paper by Duoandikoetxea and Seijo \cite{DuoandikoetxeaSeijo} and is obtained by testing $M_{{\mathcal{E}}}$ on $\chi_\delta$, the characteristic function of the ball of radius $\delta$, centered at the origin.

It remains to prove the necessary condition $\alpha\ge L(p)$ in Theorem \ref{thm:main} which by \eqref{eq:equivlower}
is equivalent to
\begin{equation}\label{eqn:maincondition}
\nu^\sharp\big((d-1)(p-2)-\alpha\big) \le (d-1)(p-1)
\end{equation}
(see \eqref{eq:typesetTheta}).
To this end it will suffice to prove that
\begin{equation}\label{eqn:penultlowerbd}
\|M_{E}\|_{L^p(w_\alpha)\to L^p(w_\alpha)}^p
\gtrsim 2^{-j(d-1)(p-1)} N(E\cap I, 2^{-j}) |I|^{(d-1)(2-p)+\alpha}
\end{equation}
for all $E\subset [1,2]$ and all $j\ge 1$ and all intervals $I\subset [1,2]$ of length $2^{-j}\le |I|\le 1$.
Indeed, using \eqref{eq:scaling} that would imply that for every ${{\mathcal{E}}}\subset (0,\infty)$ and every $R>0$ and $j\ge 1$, $2^{-j}\le |I|\le 1$,
\begin{align*}
\|M_{{{\mathcal{E}}}}\|_{L^p(w_\alpha)\to L^p(w_\alpha)}^p &\ge \|M_{{{\mathcal{E}}}_R}\|_{L^p(w_\alpha)\to L^p(w_\alpha)}^p \\&\ge
2^{-j(d-1)(p-1)} N({{\mathcal{E}}}_R\cap I, 2^{-j}) |I|^{(d-1)(2-p)+\alpha},
\end{align*}
where ${{\mathcal{E}}}_R=R^{-1}({{\mathcal{E}}}\cap [R,2R])$. By dilation invariance of the metric \eqref{eqn:metric}, $N({{\mathcal{E}}}_R\cap I, 2^{-j})=N({{\mathcal{E}}}\cap J,2^{-j})$ for $J=RI\cap [R,2R]$
and $|J|_\times = |I|_\times\approx |I|$. Taking a supremum over $j, I$ and $R>0$ and comparing to the definition \eqref{eq:LAglob} therefore gives \eqref{eqn:maincondition}.

It remains to show \eqref{eqn:penultlowerbd}.
Notice that for all $f$ supported in
\[{{\mathfrak{A}}} =\{y:1\le |y|\le 2\}\]
we have $\|f\|_{L^p(w_\alpha)}\approx \|f\|_p$. Observe that \eqref{eqn:penultlowerbd} follows from the following:
\begin{lemma}\label{lem:lowerboundgoal}
Let $E\subset [1,2]$ and $0< k\le j$, and $I\subset [1,2]$ an interval of length $2^{-k}$. Then
\begin{multline}\label{eq:lowerbdgoal}
\sup_{\substack{\|f\|_p
\le 1\\ {{\text{\rm supp }}}(f)\subset {{\mathfrak{A}}}}} 2^{k\frac \alpha p} \Big( \int_{|x|\le 2^{-k} } |M_{E} f(x)|^p dx\Big)^{\frac 1p}
\\
{\gtrsim} 2^{-j(d-1)(1-\frac 1p)} N(E\cap I, 2^{-j})^{\frac 1p} 2^{-k(\frac \alpha p+(d-1)(\frac 2p-1))} .
\end{multline}
\end{lemma}
\begin{remark}
The case $k=j/2$ in \eqref{eq:lowerbdgoal} corresponds to a standard Knapp-type example which is essentially in \cite{DuoandikoetxeaSeijo}.
\end{remark}

In the proof of Lemma \ref{lem:lowerboundgoal} we may assume that $E\cap I$ is non-empty
since otherwise there is nothing to show. In this case choose an element $a\in E\cap I$.
We further distinguish the cases $0 <k\le j/2$ and $j/2<k\le j$.
We let $\tau_{j} (I)$ be a maximal $2^{-j}$-separated set in $E\cap I$, so that \[\# \tau_{j} (I)\approx N(E\cap I, 2^{-j}). \]

\subsection{\texorpdfstring{Proof of Lemma \ref{lem:lowerboundgoal} for the case $0 < k\le j/2$}{Proof of Lemma \ref{lem:lowerboundgoal} for small k}}
Let ${\varepsilon}>0$ be small and
\[ U (a,t)= \{(x',x_d): (t-{\varepsilon} 2^{-j})^2\le |x'|^2+|x_d-a|^2 \le (t+{\varepsilon} 2^{-j})^2,\, |x'|\le 2^{-k} \} \]
which is a $2^{-j}$-neighborhood of a subset with diameter $O(2^{-k})$ of the $t$-sphere centered at $(0,a)$.
Since $ a \in I $ and $ | I | \leq 2^{-k} $, then for every $ t \in I $ the subset $ U (a,t) $ is contained in a ball centered at the origin of radius $ O ( 2^{-k} )$.

We note that \begin{equation}\label{emptyintersect} U(a,t) \cap U(a, t') =\emptyset \quad \text{ for $t,t'\in I$, $|t-t'| \ge 2^{-j}$.} \end{equation}

Let \[Q(a) = \{ (y',y_d) : |y_d-a|\le {\varepsilon}^{-1}2^{-j}, \, |y'|\le {\varepsilon}^{-1} 2^{k-j} \} .\] Note that $Q(a)$ is comparable to a box with $(d-1)$ long horizontal sides of length $2^{k-j}$ and one short vertical side of length $2^{-j}$ around the point $(0,a)$. Let $f_Q= {{\mathbbm 1}}_{Q(a)}$. We observe \[ \|f_Q\|_p {\lesssim}_{\varepsilon} 2^{(k-j)(d-1)/p} 2^{-j/p} .\]

Given ${\varepsilon}$ be small enough, we will show that for $x\in U(a,t) $
there is a subset $S(x)$ of the sphere of radius $t$ centered at $x$ which is of diameter $\approx {\varepsilon} 2^{k-j}$ and such that $S(x)\subset Q(a)$.

Let $x\in U(a,t)$ and consider the unit vectors
\[ \theta_{x,w'}= \frac{ (-x'+w', a-x_d)}{\sqrt{|x'-w'|^2+(a-x_d)^2}}.\]
Consider the subset $S(x)$ of the unit sphere, defined by \[ S(x) = \{\theta_{x,w'}: |w'|\le 2^{k-j}\} \]
Then the spherical surface measure of $S(x)$ is ${\gtrsim} 2^{(k-j)(d-1)} $ (uniformly in $x\in U(a,t)$, and $1\le t\le 2$).

We claim that for $x\in U(a,t)$
\begin{equation}
\label{eq:inclusion} \{ x+t\theta_{x,w'} :|w'| \le 2^{k-j} \} \subset Q(a)\,.\end{equation}

We let $e_d=(0,\dots,0,1)$ and let $\Pi$ be the map $(x',x_d)\mapsto x'$.
We write
\[ x+t\theta_{x,w'} -a e_d= x+t\theta_{x,0}-a e_d+ t(\theta_{x,w'}-\theta_{x,0})\] and
compute
\[ x+t\theta_{x,0}-a e_d= \frac{ (x', x_d-a)}{\sqrt{|x'|^2 +(a-x_d)^2}} (\sqrt{|x'|^2+(a-x_d)^2}-t)\] which by definition of $U(a,t)$ implies $|x+t\theta_{x,0}-a e_d|\le 2^{-j}$.

Next,
\begin{align*} &\big| \Pi(\theta_{x,w'}-\theta_{x,0})\big|= \Big|\frac{ (-x'+w') \sqrt{|x'|^2+(a-x_d)^2} +x'
\sqrt{|x'-w'|^2+(a-x_d)^2} }{ \sqrt{|x'|^2+(a-x_d)^2} \sqrt{|x'-w'|^2+(a-x_d)^2} }\Big|
\\
&\le \frac{|w'| \sqrt{|x'|^2+(a-x_d)^2} + |x'| \big| \sqrt{|x'|^2+(a-x_d)^2}-
\sqrt{|x'-w'|^2+(a-x_d)^2} \big|}{(a-x_d)^2}
\end{align*}
and clearly $|\Pi(\theta_{x,w'}-\theta_{x,0})|{\lesssim} 2^{k-j}$.
Moreover, using for $b,\tilde b>0$ the identity \[ |b^{-1/2}-\tilde b^{-1/2}|= |b-\tilde b| b^{ -1/2} \tilde b^{-1/2} (b^{1/2}+\tilde b^{1/2})^{-1} \]
we obtain
\begin{align*}
&\big|\inn{\theta_{x,w'}-\theta_{x,0}}{e_d} \big| \\&\qquad=
(a-x_d) \big| (|x'-w'|^2+(a-x_d)^2)^{-\frac12} - (|x'|^2+(a-x_d)^2)^{-\frac 12} \big|
\\
&\qquad \lesssim
\big| |x'-w'|^2 - |x'|^2 \big|
\lesssim
\big( |w'|^2+ 2|x'||w'| \big) \,.
\end{align*}
Now $|x'||w'|{\lesssim} 2^{-k}\cdot 2^{k-j}{\lesssim} 2^{-j}$ and $|w'|^2{\lesssim} 2^{2k-2j}$, and since $k\le j/2$ we also have $|w'|^2{\lesssim} 2^{-j}$. Hence
$|\inn{\theta_{x,w'}-\theta_{x,0}}{e_d} |{\lesssim} 2^{-j}$.
This means (with the appropriate choice of ${\varepsilon}>0$ in the definition of $Q(a)$) that \eqref{eq:inclusion} is proved. This implies
\[ |A_t f(x)| {\gtrsim} 2^{(k-j)(d-1)} \text{ for $x\in U(a,t)$.} \]

Now let \[ U(a)= \bigcup_{t\in \tau_j (I)} U(a,t).\]
Then by \eqref{emptyintersect},
\begin{align*} &\Big(\int_{U(a)} \Big[\sup_{s\in E\cap I} |A_s f_Q(x)|\Big]^p |x|^\alpha dx \Big)^{1/p}
\ge \Big(\sum_{t\in \tau_j(I)} \int_{U(a,t) } |A_t f_Q(x)|^p |x|^\alpha dx\Big )^{1/p}
\\&{\gtrsim} 2^{-k\alpha/p} 2^{(k-j)(d-1)} \Big(\sum_{t\in \tau_j(I)} |U(a,t) |\Big)^{1/p}\\ &{\gtrsim} 2^{-k\alpha/p} 2^{(k-j)(d-1)} N(E\cap I,2^{-j} )^{1/p} 2^{-j/p} 2^{-k(d-1)/p}
\end{align*}
and thus, with $|I|=2^{-k}$
\begin{align*} \frac{\|M_E f_Q\|_{L^p(w_\alpha)}}{\|f_Q\|_{L^p(w_\alpha) } }&{\gtrsim} \frac{ 2^{-k\alpha/p} 2^{(k-j)(d-1)} N(E\cap I,2^{-j} )^{1/p}2^{-j/p} 2^{-k(d-1)/p} }{2^{(k-j)(d-1)/p} 2^{-j/p} }
\\& {\gtrsim} 2^{-j(d-1) (1-\frac 1p)} N(E\cap I, 2^{-j} ) ^{\frac{1}{p} } |I| ^{\frac{\alpha}{p} +(d-1) (\frac 2p-1)} .
\qedhere
\end{align*}

\subsection{\texorpdfstring{Proof of Lemma \ref{lem:lowerboundgoal} for the case $j/2<k\le j$}{Proof of Lemma \ref{lem:lowerboundgoal} for large k}}
Let
\[ {{\mathcal{U}}}(a)= \{ (y',y_d): |y'|\le {\varepsilon}^{-1} 2^{k-j}, \, (a-2^{-j})^2\le|y'|^2 + |y_d|^2 \le (a+2^{-j})^2\}\]
which is a $2^{-j}$-neighborhood of a subset with diameter $O(2^{k-j})$ of the $a$-sphere centered at $(0,0)$. Let $g_{{{\mathcal{U}}}}={{\mathbbm 1}}_{{{\mathcal{U}}}(a)}.$ Then
\[ \|g_{{{\mathcal{U}}}} \|_p {\lesssim}_{\varepsilon} 2^{(k-j)(d-1)/p} 2^{-j/p}. \]

For $t\in I$, let \[ {{\mathcal{Q}}}(a,t)= \{ (x', x_d): {\varepsilon} 2^{-k-1} \le |x'| \le {\varepsilon} 2^{-k},\,|x_d-a+t| \le {\varepsilon} 2^{-j} \} \] which is comparable to a rectangle in $\{x:|x|\approx {\varepsilon} 2^{-k}\}$, with $(d-1)$ long horizontal sides of length $\approx {\varepsilon} 2^{-k}$ and one short vertical side of length ${\varepsilon} 2^{-j}$.
Observe \begin{equation}
\label{emptyintersect2} {{\mathcal{Q}}}(a,t) \cap {{\mathcal{Q}}}(a, t') =\emptyset \text{ for $t,t'\in I$, $|t-t'| \ge 2^{-j}$.}
\end{equation}

We show that
\begin{equation}
\label{eq:lowerbd} A_t g_{{{\mathcal{U}}}} (x){\gtrsim} 2^{(k-j)(d-1)} \,\,\text{ for $x\in {{\mathcal{Q}}}(a,t)$.}
\end{equation}
This is a calculation in
\cite[\S5]{RoosSeeger} which for convenience of the reader we reproduce with our current parameters.

Let $x\in {{\mathcal{Q}}}(a,t)$, ${\omega}=({\omega}',{\omega}_d)\in S^{d-1}$, with $|{\omega}'|\le {\varepsilon} 2^{k-j} $ and ${\omega}_d=\sqrt{1-|{\omega}'|^2}$. We wish to show that
$x+t{\omega} \in {{\mathcal{U}}}(a)$. To see this we compute
\begin{align*}
|x+t{\omega}|^2 &= |x'|^2 + x_d^2 + t^2+ 2t \inn{x'}{{\omega}'} + 2tx_d {\omega}_d \\
& = |x'|^2 + (x_d+t)^2 + 2t \inn{x'}{{\omega}'}+2tx_d (\sqrt{1-|{\omega}'|^2}-1) .
\end{align*}
Since $|x_d + t - a|\le {\varepsilon} 2^{-j} $ and $|x'|^2 \le {\varepsilon}^2 2^{-2k}\le {\varepsilon}^2 2^{-j}$, we get
\[ ||x'|^2 + (x_d+t)^2 - a^2| \le 6 {\varepsilon} 2^{-j}, \]
\[ |2t\inn{x'}{{\omega}'}| \le 4 |x'| |{\omega}'| \le 4 {\varepsilon} 2^{-j}, \] and
\[ |2t x_d (\sqrt{1-|{\omega}'|^2}-1)| \le 2 (|t-a|+{\varepsilon} 2^{-j} ) |{\omega}'|^2 \le 2 (2^{-k}+{\varepsilon} 2^{-j}) {\varepsilon} 2^{k-j} \le 4{\varepsilon}2^{-j}.\]
Hence
\[ ||x+t{\omega}|^2 - a^2| \le 14 {\varepsilon} 2^{-j}, \]
and thus if ${\varepsilon}\le 10^{-2}$ we have $||x+t{\omega} |-a|\le 2^{-j}$.
Also, since $k\ge j/2$, \[|x'+t{\omega}'| \le |x'| + 2 |{\omega}'| \le 2^{-k} + 2^{k-j} \le 2^{k-j+1}\] so that altogether $x+t{\omega} \in {{\mathcal{U}}}(a)$. This establishes
\eqref{eq:lowerbd}.

Now let ${{\mathcal{Q}}}(a) = \cup_{t\in \tau_j(I)} {{\mathcal{Q}}}(a,t)$. Then by \eqref{eq:lowerbd}
\begin{align*} &\Big(\int_{{{\mathcal{Q}}}(a)} \Big[\sup_{s\in E\cap I} |A_s g_{{\mathcal{U}}}(x)|\Big]^p |x|^\alpha dx \Big)^{1/p}
\\&\quad\ge \Big(\sum_{t\in \tau_j (I)} \int_{{{\mathcal{Q}}}(a,t) } |A_t g_{{{\mathcal{U}}}} (x)|^p |x|^\alpha dx\Big )^{1/p}
\\&\quad{\gtrsim} 2^{-k\alpha/p} 2^{(k-j)(d-1)} \Big(\sum_{t\in \tau_j(I)} |{{\mathcal{Q}}}(a,t) |\Big)^{1/p}
\\
&\quad{\gtrsim} 2^{-k\alpha/p} 2^{(k-j)(d-1)} N(E\cap I,2^{-j} )^{1/p} 2^{-j/p} 2^{-k(d-1)/p}
\end{align*}
and
\begin{align*} \frac{\|M_E g_{{{\mathcal{U}}}} \|_{L^p(w_\alpha)}}{\|g_{{{\mathcal{U}}}} \|_{L^p(w_\alpha) } }&{\gtrsim} \frac{ 2^{-k\alpha/p} 2^{(k-j)(d-1)} N(E\cap I,2^{-j} )^{1/p}2^{-j/p} 2^{-k(d-1)/p} }{2^{(k-j)(d-1)/p} 2^{-j/p} }
\\& {\gtrsim} 2^{-j(d-1) (1-\frac 1p)} N(E\cap I, 2^{-j} ) ^{\frac{1}{p} } |I| ^{\frac{\alpha}{p} +(d-1) (\frac 2p-1)} . \qedhere
\end{align*}

\section{Upper bounds}\label{sec:upperbounds}

In this section we prove the upper bounds in Theorem \ref{thm:main}, i.e. that $M_{{\mathcal{E}}}$ is bounded if
$\alpha< (d-1)(p-1)-\beta$ and
\begin{equation}\label{eqn:mainasm}
\nu^\sharp_{{\mathcal{E}}}((d-1)(p-2)-\alpha)< (d-1)(p-1)
\end{equation}
(note here we used the equivalent condition from \eqref{eq:Thetadef}).
We may also assume that $\alpha < 0$
since the case $\alpha\ge 0$ was already handled in \cite{DuoandikoetxeaSeijo}.
For the same reason we may also assume $p\le 1 + \frac{1}{d-1}\le 2$.

As a consequence of the above and \S\ref{sec:localglobal}, it will now suffice to show boundedness of $M_{{{\mathcal{E}}}}$ on $L^p(w_\alpha)$ under the additional assumption ${{\mathcal{E}}}=E\subset [1,2]$. To be precise, we need to prove the following

\begin{proposition}\label{prop:local} Let $p\in [1,2]$ and $\alpha\le 0$.
Assume that $E\subset [1,2]$ is non-empty and there exists $\varepsilon>0$ and a constant $A\ge 1$ such that
\begin{equation}\label{mainassu}
\sup_{0<\delta<1} \sup_{\delta\le |I|\le 1} N(E\cap I, \delta) \delta^{(d-1)(p-1)-\varepsilon}
|I|^{\alpha+(d-1)(2-p)} \le A^p.
\end{equation}
Then $\|M_{E}\|_{L^p(w_\alpha)\to L^p(w_\alpha)} \lesssim A$ with the implicit constant only depending on $p,\alpha,d$.
\end{proposition}
\begin{remark}
Note that since $E\subset [1,2]$ here, it makes no difference whether we use the Euclidean metric or the metric \eqref{eqn:metric} in \eqref{mainassu}.
\end{remark}

Proposition \ref{prop:local} completes the proof of Theorem \ref{thm:main}. Since we may
assume $\alpha<0$ and $p\le 2$, \eqref{eqn:mainasm} implies that there exists $\varepsilon>0$ and $A\ge 1$ so that for every $R>0$,
\[ \sup_{0<\delta<1} \sup_{\delta\le |I|\le 1} N({{\mathcal{E}}}_R\cap I, \delta) \delta^{(d-1)(p-1)-\varepsilon}
|I|^{\alpha+(d-1)(2-p)} \le A^p.\]
By Proposition \ref{prop:local} applied to each $E={{\mathcal{E}}}_R$ we then obtain
\[ \|M_{{{\mathcal{E}}}_R}\|_{L^p(w_\alpha)\to L^p(w_\alpha)} \lesssim A, \]
for every $R>0$, which by Lemma \ref{lem:global} implies that $M_{{{\mathcal{E}}}}$ is bounded on $L^p(w_\alpha)$ as required.

\subsection{Proof of Proposition \ref{prop:local}}
The arguments here are close to \cite{DuoandikoetxeaVega,DuoandikoetxeaSeijo}.
We first observe that the assumption \eqref{mainassu} implies
\[ p>p_\beta=1+\tfrac{\beta}{d-1} \]
by considering the case $|I|=1$ in the supremum and also
\[ \alpha>1-d \]
by considering the case $|I|=\delta$ in the supremum. We split
\begin{equation}
\label{eq:lowhigh}f=f_{\mathrm{low}} + f_0 + f_{\mathrm{high}},
\end{equation} where $f_{\mathrm{low}} $ lives on $ \{y:|y|<1/4\}$, $f_0$ lives on $\{y: 1/2\le |y|\le 4\}$, and
$f_{\mathrm{high}} $ lives on $ \{y:|y|>4\}$. Then $M_{E} f_{\mathrm{low}}(x)=0 $ for $|x|> 9/4$ and for $|x|<1/2$.
Thus for $f_{\mathrm{low}}$ the $L^p(w_\alpha)$ bound follows from the unweighted bound. Indeed,
\begin{align}\label{eq:MEflow}\|M_{E} f_{\mathrm{low}}\|_{L^p(w_\alpha)} &{\lesssim}_\alpha
\Big(\int_{\frac 12\le |x|\le \frac 94} |M_{E} f_{\mathrm{low}} (x)|^p dx\Big)^{\frac 1p} \\&\notag{\lesssim} \|f_{\mathrm{low}} \|_p {\lesssim} \|f_{\mathrm{low}} \|_{L^p(w_\alpha)} {\lesssim} \|f \|_{L^p(w_\alpha)},
\end{align}
where we have used $\alpha\le 0$ and $p>p_\beta$.

We split $f_{\mathrm{high}}=\sum_{k=2}^\infty f^k$ where $f^k=f{{\mathbbm 1}}_{2^k<|y|\le 2^{k+1}}$. Let \[{{\mathfrak{A}}}_k =\{x: 2^k-2\le |x|\le 2^{k+1}+2\}. \]
Then $M_{E} f^k(x)=0$ for $x\in {{\mathfrak{A}}}_k^\complement$.
Again, for all $\alpha\in {{\mathbb{R}}}$ we can use the unweighted inequality to estimate
\begin{align} \label{eq:MEfhigh} \|M_{E} f_{\mathrm{high}}\|_{L^p(w_\alpha)} &\le
\Big(\int_{|x|\ge 2} \Big| \sum_{k=2}^\infty {{\mathbbm 1}}_{{{\mathfrak{A}}}_k} (x) M_{E} f^k(x)\Big |^p |x|^\alpha dx \Big)^{\frac 1p}
\\\notag &{\lesssim}
\Big(\sum_{k=2}^\infty 2^{k\alpha} \int |M_{E} f^k(x) |^p dx \Big)^{\frac 1p}
\\ \notag &{\lesssim}
\Big(\sum_{k=2}^\infty 2^{k\alpha} \int | f^k(x) |^p dx \Big)^{\frac 1p} {\lesssim} \|f\|_{L^p(w_\alpha)}.
\end{align}

The estimate for $M_{E} f_0$ is more interesting; only here the weight plays an essential role.
Let $\phi_0\in C^\infty_c$ be supported in $\{x:|x|<1/4\}$ such that $\int \phi_0(x) x^{\iota}\,dx=0$ for all multiindices $\iota\in \mathbb{N}_0^d$ with $1\le \sum_{j=1}^d\iota_j\le N_0$ where $N_0 \gg\frac{d-1}{2}$. For $j\ge 1$ let $\phi_j(x)= 2^{jd}\phi_0(2^jx)- 2^{(j-1)d}\phi_0(2^{j-1} x)$ and set \[K^j(x) =\sigma*\phi_j(x)\] and
$K_t^j(x)= t^{-d} K^j(t^{-1}x)$. We then have the decomposition
$\sigma_t=\sum_{j=0}^\infty K_t^j$ which has the behavior of a dyadic frequency decomposition but with strong spatial localization. Set $A^j_t f= f*K^j_t$.

\begin{lemma}\label{lem:k=j}
Let $\alpha\ge 1-d.$ Then for $j\ge 0$,
\begin{equation}
\label{j-ineq} \Big(\int_{|x|\le 2^{-j} } \sup_{1\le t\le 2} \big |A_t^j f_0(x) \big|^p|x|^\alpha dx\Big)^{\frac 1p}
\lesssim 2^{-j \frac{d-1+\alpha}p} \|f_0\|_p .
\end{equation}
\end{lemma}

\begin{proof}
For $2^{j-3}\le \nu \le 2^{j+5}$ let ${{\mathbbm 1}}_\nu$ be the characteristic function of the thin annulus ${{\mathcal{R}}}_\nu= \{ y: \big| |y|-\nu 2^{-j} \big |\le 2^{-j} \}$ and $J_\nu=[(\nu-2)2^{-j}, (\nu+3)2^{-j}]$. Let $B_j=\{x:|x|\le 2^{-j}\}$. Then
\begin{equation} \begin{aligned}
&\Big(\int_{B_j}\sup_{1\le t\le 2} |A^j_t f_0(x) |^p |x|^\alpha dx\Big)^{\frac 1p}\\&=
\Big(\int_{B_j}\sup_{1\le t\le 2}\Big |A^j_t [\sum_\nu (f_0{{\mathbbm 1}}_\nu) ](x) \Big |^p |x|^\alpha dx\Big)^{\frac 1p}
\\&{\lesssim}
\Big( \sum_\nu \int_{B_j}\sup_{t\in J_\nu} \big |A^j_t (f_0{{\mathbbm 1}}_\nu) (x) \big|^p |x|^\alpha dx\Big)^{\frac 1p}.
\end{aligned}\end{equation}
Now use the pointwise bound \begin{equation}\label{eq:ftc} \sup_{t\in J_\nu} \big |A^j_t (f_0{{\mathbbm 1}}_\nu) (x) \big |\le |A^j_{\nu2^{-j} }(f_0{{\mathbbm 1}}_\nu)(x)|+\int_{J_\nu} \Big|\tfrac{d}{ds} A^j_s (f_0{{\mathbbm 1}}_\nu) (x) \Big| ds\,.\end{equation}
We interpolate inequalities for $\alpha=1-d$ and for $\alpha=0$.

For $\alpha=1-d $ we have, for $1\le p<\infty$,
\begin{subequations}
\begin{align}
&\Big( \int_{B_j}\Big| A^j_{2^{-j}\nu} (f_0{{\mathbbm 1}}_\nu) (x) \Big |^p |x|^{1-d} dx\Big)^{\frac 1p}{\lesssim} \|f_0{{\mathbbm 1}}_\nu\|_p,
\\
&\Big( \int_{B_j}\Big[\int_{J_\nu} \big|\tfrac d{ds} A^j_{s} (f_0{{\mathbbm 1}}_\nu) (x) \big |ds \Big]^p |x|^{1-d} dx\Big)^{\frac 1p} {\lesssim} \|f_0{{\mathbbm 1}}_\nu\|_p,
\end{align}
\end{subequations}
which follow by interpolation between $p=1$ and the corresponding inequality for $p=\infty$. For $p=1$ we use that $\|K^j_t\|_\infty+ 2^{-j} \|\frac{d}{dt} K^{j}_t\|_\infty = O(2^j) $ and $\int_{B_j}|x|^{1-d} dx=O(2^{-j})$.

For $\alpha=0$ and $1\le p\le 2$ we have
\begin{subequations}
\begin{align}
&\Big( \int_{B_j}\big| A^j_{2^{-j}\nu} [f_0{{\mathbbm 1}}_\nu ](x) \big |^p dx\Big)^{\frac 1p}{\lesssim} 2^{-j(d-1)(1-\frac 1p)} \|f_0{{\mathbbm 1}} _\nu\|_p,
\\
&\Big( \int_{B_j}\Big(\int_{J_\nu} \big|\tfrac d{ds} A^j_{s} [f_0{{\mathbbm 1}}_\nu] (x) \big |ds \Big)^p dx\Big)^{\frac 1p} {\lesssim} 2^{-j(d-1)(1-\frac 1p)}\|f_0{{\mathbbm 1}}_\nu\|_p.
\end{align}
\end{subequations}
Again this follows by interpolation between $p=1$ and $p=2$. For $p=1$ we just use $\|K^j_t\|_1+2^{-j} \|\frac{d}{dt} K^j_t\|_1=O(1)$. For $p=2$ we use $|\widehat \sigma(\xi)|{\lesssim} (1+|\xi|)^{-\frac{d-1}{2}} $ and thus, with $\frac{d-1}{2}\ll N_0\ll N$, \begin{subequations}
\begin{equation}\label{eq:FTcalc} |\widehat{K^j_1} (\xi)|{\lesssim} (2^{-j} |\xi|) ^{N_0 } (1+ (2^{-j}|\xi|))^{-N} (1+|\xi|)^{-\frac{d-1}{2}} {\lesssim} 2^{-j\frac{d-1}{2}}. \end{equation}
We have $\frac{d}{dt} \widehat{K^j_t}(\xi)= \inn{2^{-j}\xi}{\nabla \phi(2^{-j}t\xi)}\widehat \sigma(t\xi) + \phi(2^{-j} t\xi) \inn{\xi}{\nabla \widehat \sigma }(t \xi)$ and obtain
\begin{equation}\label{eq:FTcalcder} \big|\tfrac{d}{dt} \widehat{K^j_t}(\xi) \big | {\lesssim} 2^{-j\frac{d-3}{2}} , \quad \text{ $t\approx 1$.}\end{equation}
\end{subequations}
Since we also integrate over $J_\nu$ the interpolation with change of weight (\cite{SteinWeiss-weights}) together with \eqref{eq:ftc} leads to
\[
\Big( \int_{B_j}\sup_{t\in J_\nu} \big|A^j_t [f_0{{\mathbbm 1}}_\nu ](x) \big|^p |x|^\alpha dx\Big)^{\frac 1p} {\lesssim} 2^{-j\frac{d-1+\alpha}p}\|f_0{{\mathbbm 1}}_\nu\|_p
\] and then also
\[\Big( \int_{B_j}\big|\sup_{ 1/4\le t\le 4}\big |A^j_t f_0 (x) \big|^p |x|^\alpha dx\Big)^{\frac 1p} {\lesssim} 2^{-j\frac{d-1+\alpha}p}\|f\|_p.\qedhere
\]

\end{proof}

\begin{lemma}\label{lem:k<j} Let $E\subset[1,2]$. For $0<k\le j$,
\begin{multline} \label{jk-ineq} \Big(\int_{2^{-k-1}\le |x|\le 2^{-k}} \sup_{t\in E} |A_t^j f_0(x)|^p|x|^\alpha dx\Big)^{\frac 1p} \\
\lesssim 2^{-j\frac{d-1}{p'} } \sup_{\substack{I\subset [1,2] \\ |I|\approx 2^{-k} } }\big[ N(E\cap I,2^{-j} ) 2^{-k((d-1)( 2-p)+\alpha) } \big]^{\frac 1p} \|f_0\|_p\,.\end{multline}
\end{lemma}
\begin{proof}

We let $I_{k,\mu} =[ 2^{-k} (\mu-1) , 2^{-k}( \mu+2)].$ For each $\mu$ with $E\cap I_{k,\mu} \neq \emptyset$, denote by $\Gamma_{k,\mu,j}$ a collection of intervals of length $2^{-j}$ covering $E\cap I_{k,\mu}$ such that $N(E\cap I_{k,\mu} , 2^{-j}) \approx \# \Gamma_{k,\mu,j} $. If $J\in \Gamma_{k,\mu,j}$ we let $t_J$ be the center of $J$.

We first prove that for $1\le p\le 2$, $k<j$, $B_k=\{x:|x|\le 2^{-k}\}$,
\begin{multline}\label{eq:mainunweighted}
\Big( \int_{B_k\setminus B_{k+1} } \sup_{t\in E\cap I_{k,\mu}} |A_t^j f_0 |^p dx\Big)^{\frac 1p}\\ {\lesssim} 2^{-j(d-1)(1-\frac 1p)} 2^{-k(d-1)(\frac 2p-1)} N(E\cap I_{k,\mu} , 2^{-j}) ^{\frac 1p} \|f_0\|_p.
\end{multline}
We use the pointwise estimate
\begin{multline} \label{eq:ptwest} \sup_{t\in E\cap I_{k,\mu}} |A_t^j f_0(x) |
\\ \le \Big(\sum_{J\in {\Gamma}_{k,\mu,j}} |A_{t_J}^j f_0(x) |^p \Big)^{\frac 1p}
+ \Big(\sum_{J\in {\Gamma}_{k,\mu,j} }\Big(\int_{|s|\le 2^{-j}} |\tfrac{d}{ds}A_{t_J+s}^j f_0(x) |ds \Big) ^p \Big)^{\frac 1p}.
\end{multline}
By the integral Minkowski inequality we need to show
\begin{subequations}
\begin{multline}\label{eq:jkmu-1}
\Big( \sum_{J\in {\Gamma}_{k,\mu,j}} \int_{B_k\setminus B_{k+1} } |A_{t_J} ^j f_0 |^p dx\Big)^{\frac 1p}
\\ {\lesssim} 2^{-j(d-1)(1-\frac 1p)} 2^{-k(d-1)(\frac 2p-1)} N(E\cap I_{k,\mu} , 2^{-j}) ^{\frac 1p} \|f_0\|_p
\end{multline}
and
\begin{multline} \label{eq:jkmu-2}
2^{-j} \sup_{|s|\le 2^{-j}}
\Big( \sum_{J\in {\Gamma}_{k,\mu,j}} \int_{B_k\setminus B_{k+1} } |\tfrac d{ds} A_{t_J+s} ^j f_0 |^p dx\Big)^{\frac 1p}
\\ {\lesssim} 2^{-j(d-1)(1-\frac 1p)} 2^{-k(d-1)(\frac 2p-1)} N(E\cap I_{k,\mu} , 2^{-j}) ^{\frac 1p} \|f_0\|_p\,.
\end{multline}
\end{subequations}

We give the proof of \eqref{eq:jkmu-1} and omit the completely analogous proof of \eqref{eq:jkmu-2}.
These inequalities follow from the cases $p=1$ and $p=2$ by interpolation. For $p=1$ we estimate
\begin{align*} \int_{B_k} |A_{t_J} ^j f_0(x) | dx \le \int |f_0(y)| \int_{B_k} |K_{t_J}^j (x-y)|dx \, dy{\lesssim} 2^{-k(d-1)} \|f_0\|_1
\end{align*}
where we have used that $\{x\in B_k: ||x-y|-t_J|{\lesssim} 2^{-j} \}$ has measure $O(2^{-k(d-1)} 2^{-j})$ and $\|K_t^j\|_\infty =O(2^{j})$ for $1\le t\le 2$. \eqref{eq:jkmu-1} follows for $p=1$ by summing over the intervals in ${\Gamma}_{k,\mu,j}$. For $p=2$
we just use $\|A^j_{t_J}\|_{L^2\to L^2}=O( 2^{-j(d-1)/2} )$ (see \eqref{eq:FTcalc}) and \eqref{eq:jkmu-1} for $p=2$ follows.
Analogous arguments apply to \eqref{eq:jkmu-2} and thus one gets
\eqref{eq:mainunweighted}.

Let \[f_{k,\mu}(y)=f_0(y) {{\mathbbm 1}}_{2^{-k}\mu\le |y|\le 2^{-k}(\mu+1)} (y)\] and observe that \[ A^j_t f(x) =A^j_t\Big[ \sum_{m=\mu-2}^{\mu+2} f_{k,m} \Big ] (x) \quad \text{ for $|x|\le 2^{-k}$ and $t\in I_{k,\mu} $.}
\]
Moreover $w_\alpha(x)\approx 2^{-k\alpha}$ for $x\in B_k\setminus B_{k+1}$. Hence,
\begin{align*}
&\Big(\int_{B_k\setminus B_{k+1}} \sup_{t\in E} |A_t^j f_0(x)|^p|x|^\alpha dx\Big)^{\frac 1p} \\
&{\lesssim} 2^{-k\frac{\alpha}{p}} \Big( \int_{B_k} \sum_{\mu} \sup_{t\in E\cap I_{k,\mu} } \Big| A^j_t \big [\sum_{m=\mu-2}^{\mu+2} f_{k,m} \big] (x) \Big|^p dx\Big)^{\frac 1p}
\\
&{\lesssim} 2^{-k\frac{\alpha}{p}} 2^{-j(d-1)(1-\frac 1p)} 2^{-k(d-1)(\frac 2p-1)} \sup_\mu N(E\cap I_{k,\mu} , 2^{-j} )^{\frac 1p} \Big(\sum_\mu \Big\|\sum_{m=\mu-2}^{\mu+2} f_{k,m}\Big\|_p^p\Big)^{\frac 1p}
\\
&{\lesssim} 2^{-j(d-1)(1-\frac 1p)} 2^{-k(\frac{\alpha}{p}+(d-1)(\frac 2p-1))} \sup_\mu N(E\cap I_{k,\mu} , 2^{-j} )^{\frac 1p} \|f\|_p
\end{align*}
and the lemma is proved.
\end{proof}

To conclude the proof of Proposition \ref{prop:local} we set \[M^j_{E}f(x)= \sup_{t\in E} |A_t^j f(x)|. \]
By Lemma \ref{lem:k=j} and Lemma \ref{lem:k<j},
\begin{multline*}
\|M^j_{E} f_0\|_{L^p(w_\alpha)} \\
\begin{aligned} &{\lesssim} 2^{-j(d-1)(1-\frac 1p)} \sum_{k=0}^j\sup_{|I|=2^{-k}} N(E\cap I, 2^{-j})^{\frac 1p} 2^{-k((d-1)(\frac 2p-1)+\frac \alpha p)} \|f_0\|_p
\\&{\lesssim} (1+j) 2^{-j\frac{\varepsilon} p} A \|f\|_{L^p(w_\alpha)}
\end{aligned}
\end{multline*}
where in the last estimate we have used the main hypothesis \eqref{mainassu}.
Finally, we may sum in $j$ and then combine the resulting estimate for $M_{E}f_0$ with the estimates
\eqref{eq:MEflow} for $M_{E} f_{\mathrm{low}}$ and \eqref{eq:MEfhigh} for $M_{E} f_{\mathrm{high}}$ to complete the proof of Proposition \ref{prop:local}. \qed

\end{document}